\documentclass[11pt]{amsart}
\usepackage{url, 
  amssymb,setspace, mathrsfs,fontenc, comment}
 \usepackage[alphabetic]{amsrefs}
\usepackage{fullpage} 
\usepackage{color}
\usepackage{tikz-cd}
\usepackage[all]{xy}
\usepackage{amsmath}

\usepackage{url, 
  amssymb,setspace, mathrsfs,fontenc}
\usepackage{amsrefs}
\usepackage{graphicx}
\usepackage[mathcal]{euscript}
\usepackage{verbatim}
\usepackage{hyperref}
\hypersetup{backref=true}
\usepackage{mathtools}
\usepackage{tikz}
\usetikzlibrary{chains}

\tikzset{node distance=2em, ch/.style={circle,draw,on chain,inner sep=2pt},chj/.style={ch,join},every path/.style={shorten >=4pt,shorten <=4pt},line width=1pt,baseline=-1ex}

\newtheorem{thm}{Theorem}
\newtheorem{lem}[thm]{Lemma}
\newtheorem{prop}[thm]{Proposition}

\newtheorem{cor}[thm]{Corollary}

\theoremstyle{remark}

\theoremstyle{definition}

\DefineSimpleKey{bib}{myurl}
\newcommand\myurl[1]{\url{#1}}
\BibSpec{webpage}{
  +{}{\PrintAuthors} {author}
  +{,}{ \textit} {title}
  +{}{ \parenthesize} {date}
  +{,}{ \myurl} {myurl}
}
\usepackage{arydshln}

\newcommand{\nc}{\newcommand}

\nc{\ssec}{\subsection}

\nc{\on}{\operatorname}

\nc{\sE}{\mathscr{E}}
\nc{\sF}{\mathscr{F}}
\nc{\sL}{\mathscr{L}}
\nc{\sD}{\mathscr{D}}
\nc{\sA}{\mathscr{A}}

\nc{\cC}{\mathcal{C}}
\nc{\cG}{\mathcal{G}}
\nc{\cV}{\mathcal{V}}
\nc{\cK}{{k(\!(s)\!)}}
\nc {\K}{\mathcal{K}}

\nc{\cE} {\mathcal{E}}
\nc{\Kl}{\mathrm{Kl}}
\nc{\cO}{\mathcal{O}}
\nc{\cF}{\mathcal{F}}
\nc{\cZ}{\mathcal{Z}}
\nc{\bcZ}{\overline{\mathcal{Z}}}
\nc{\cD}{\mathcal{D}}
\nc{\cDt}{\mathcal{D}^\times}
\nc{\cH}{\mathcal{H}}
\nc{\bZ}{\mathbb{Z}}
\nc{\bQ}{\mathbb{Q}}
\nc{\bR}{\mathbb{R}}
\nc{\bC}{\mathbb{C}}
\nc{\bQl}{\overline{\mathbb{Q}}_\ell}
\nc{\bQlt}{\bQl^\times} 
\nc{\FG}{\mathrm{FG}}
\nc{\dR}{\mathrm{dR}}
\nc{\uG}{\underline{G}}
\nc{\cU}{\mathcal{U}}
\nc{\rat}{\mathrm{rat}}
\nc{\Hyp}{\mathrm{Hyp}}
\nc{\Lie}{\mathrm{Lie}}
      
\nc{\fF}{\mathfrak{F}}
\nc{\fB}{\mathfrak{B}}
\nc{\fZ}{\mathfrak{Z}}
\nc{\fx}{\mathfrak{x}}
\nc{\fy}{\mathfrak{y}}
\nc{\fb}{\mathfrak{b}}
\nc{\fk}{\mathfrak{k}}
\nc{\fI}{\mathfrak{i}}
\nc{\fj}{\mathfrak{j}}
\nc{\fg}{\mathfrak{g}}
\nc{\fu}{\mathfrak{u}}
\nc{\fl}{\mathfrak{l}}
\nc{\fn}{\mathfrak{n}}
\nc{\cP}{\mathcal{P}}
\nc{\ft}{\mathfrak{t}}
\nc{\fz}{\mathfrak{z}}
\nc{\fc}{\mathfrak{c}}
\nc{\fh}{\mathfrak{h}}
\nc{\fp}{\mathfrak{p}}
\nc{\bone}{\mathbf{1}}
\nc{\tg}{\mathtt{g}}
\nc{\hfg}{\widehat{\fg}}
\nc{\ch}{\check{\fh}}
\nc{\hg}{\widehat{\mathfrak{g}}}
\nc{\gO}{\mathfrak{g}[\![t]\!]}
\nc{\Ug}{\widehat{U}(\mathfrak{g})}
\nc{\dl}{/\!\!/}

\nc{\bGm}{\mathbb{G}_m}
\nc{\bGa}{\mathbb{G}_a}
\nc{\bL}{\mathbf{L}}
\nc{\bK}{\mathbf{K}}
\nc{\bJ}{\mathbf{J}}
\nc{\bI}{\mathbf{I}}
\nc{\bV}{\mathbb{V}}
\nc{\bM}{\mathbb{M}}
\nc{\bP}{\mathbb{P}}
\nc{\bA}{\mathbb{A}}
\nc{\bN}{\mathbb{N}}

\nc {\Q}{\mathrm{Q}}
\nc{\diag}{\mathrm{diag}}
\nc{\ev}{\mathrm{ev}}
\nc{\Res}{\mathrm{Res}}
\nc{\Fl}{\mathcal{F}\ell}
\nc{\Ad}{\mathrm{Ad}}
\nc{\ad}{\mathrm{ad}}
\nc{\pr}{\mathrm{pr}}
\nc{\Sl}{\mathfrak{sl}}
\nc{\gl}{\mathfrak{gl}}
\nc{\ra}{\rightarrow}
\nc{\tra}{\twoheadrightarrow}
\nc{\hra}{\hookrightarrow}
\nc{\quo}{\mathopen{ /\!/}}
\nc{\GL}{\mathrm{GL}}
\nc{\SL}{\mathrm{SL}}
\nc{\Sp}{\mathrm{Sp}}
\nc{\SO}{\mathrm{SO}}
\nc{\so}{\mathfrak{so}}
\nc{\PGL}{\mathrm{PGL}}
\nc{\Bun}{\mathrm{Bun}}
\nc{\supp}{\mathrm{supp}}
\nc{\bgamma}{\bar{\gamma}}
\nc{\I}{\mathrm{I}}
\nc{\II}{\mathrm{II}}
\nc{\III}{\mathrm{III}}
\nc{\ab}{\mathrm{ab}}
\nc{\td}{\mathrm{d}}
\nc{\Ht}{\mathrm{ht}}

\nc         {\rar}[1]       {\stackrel{#1}{\longrightarrow}}

\nc{\fa}{\mathfrak{a}}
\nc{\Hit}{\mathrm{Hit}}

\nc{\RS}{\mathrm{RS}}
\nc{\Loc}{\mathrm{Loc}}
\nc{\tLoc}{\widetilde{\mathrm{Loc}}}
\nc{\reg}{\mathrm{reg}}
\nc{\im}{\mathrm{Im}}

\nc{\tp}{\mathfrak{p}}
\nc{\cA}{\mathcal{A}}
\nc{\cY}{\mathcal{Y}}

\nc{\opp}{\mathrm{opp}}
\nc{\Ind}{\mathrm{Ind}}
\nc{\sAn}{\mathrm{can}}
\nc{\Vac}{\mathrm{Vac}}
\nc{\Op}{\mathrm{Op}}
\nc{\Lg}{\check{\fg}}
\nc{\cDelta}{\check{\Delta}}
\nc{\cPhi}{\check{\Phi}}
\nc{\LV}{\check{V}}
\nc{\Lh}{\check{h}}
\nc{\LG}{\check{G}}
\nc{\cT}{\check{T}}
\nc{\ct}{\check{\ft}}
\nc{\cB}{\check{B}}
\nc{\cb}{\check{\fb}}
\nc{\cN}{\check{N}}
\nc{\cn}{\check{\fn}}
\nc{\Spec}{\mathrm{Spec}}
\nc{\End}{\mathrm{End}}
\nc{\crho}{\check{\rho}}
\nc{\rX}{\mathring{X}}
\nc{\ru}{\mathring{u}}

\nc{\sW}{\mathscr{W}}
\nc{\sH}{\mathscr{H}}
\nc{\sV}{\mathscr{V}}
\nc{\geom}{\mathrm{geom}}
\nc{\Irr}{\mathrm{Irr}}
\nc{\fm}{\mathfrak{m}}
\nc{\aff}{\mathrm{aff}}
\nc{\Aut}{\mathrm{Aut}}
\nc{\cJ}{\mathcal{J}}
\nc{\fs}{\mathfrak{s}}
\nc{\Stab}{\mathrm{Stab}}
\nc{\tw}{{\widetilde{w}}}
\nc{\gen}{\mathrm{gen}}
\nc{\genn}{\mathrm{genn}}
\nc{\sss}{\mathrm{ss}}
\nc{\spp}{\mathfrak{sp}}
\nc{\Hom}{\mathrm{Hom}}
\nc{\bm}{\mathbf{m}}
\nc{\HG}{\mathcal{HG}}
\nc{\Gal}{\mathrm{Gal}}
\nc{\Sym}{\mathrm{Sym}}
\nc{\rank}{\mathrm{rank}}

\nc{\tP}{\mathtt{P}}
\nc{\tL}{\mathtt{L}}
\nc{\tU}{\mathtt{U}}

\nc{\tW}{\widetilde{W}}
\nc{\Hk}{\on{Hk}}
\nc{\cL}{\mathcal{L}}
\nc{\talpha}{\widetilde{\alpha}}
\nc{\tQ}{{\widetilde{Q}}}
\nc{\ochi}{\overline{\chi}}
\nc{\tdelta}{\widetilde{\Delta}}
\nc{\wt}{\mathrm{wt}}
\nc{\fQ}{\mathfrak{Q}}

\nc{\Rep}{\mathrm{Rep}}
\nc{\Conn}{\mathrm{Conn}}
\nc{\Hecke}{\mathrm{Hecke}}
\nc{\Gr}{\mathrm{Gr}}
\nc{\GR}{\mathrm{GR}}
\nc{\IC}{\mathrm{IC}}
\nc{\Std}{\mathrm{Std}} 
\nc{\Db}{\mathrm{D}^{\mathrm{b}}}
\nc{\tr}{\mathrm{tr}}
\nc{\gr}{\mathrm{gr}}
\nc{\Fun}{\mathrm{Fun}~}

\newcommand{\quash}[1]{}

\AtEndDocument{\bigskip{\footnotesize

\textsc{School of Mathematics, University of Minnesota, Twin cities, Minneapolis, MN 55455 } \par
\textit{E-mail address}: \texttt{lyi@umn.edu} \par
}}

\begin{document} 
\renewcommand{\thepart}{\Roman{part}}

\renewcommand{\partname}{\hspace*{20mm} Part}

\begin{abstract}
We show that the Frenkel-Gross connection on $\bGm$ is physically rigid as $\check{G}$-connection, thus confirming the de Rham version of a conjecture of Heinloth-Ng\^o-Yun. The proof is based on the construction of the Hecke eigensheaf of a connection with only generic oper structure, using the localization of Weyl modules.
\end{abstract}

\title{On the Physically Rigidity of Frenkel-Gross Connection} 
\author{Lingfei Yi} 
\date{\today} 
\maketitle

\tableofcontents

\section{Introduction}
\subsection{The Frenkel-Gross connection}
Let $G$ be a complex simply-connected simple algebraic group over $\bC$. Let $\LG$ be the Langlands dual group of $G$, which is of adjoint type. Let $\Lg$ be the Lie algebra of $\LG$. Fix a maximal torus $\cT$ of $\LG$ and a Borel subgroup $\cB$ containing $\cT$. Denote the corresponding simple roots, positive roots and roots of $\LG$ by $\cDelta\subset\cPhi^+\subset\cPhi$. For each root $\alpha\in\check{\Phi}$ of $\Lg$, we fix a basis $E_\alpha\in\Lg_\alpha$ of the root subspace. Thus $N=\sum_{\alpha\in\cDelta}E_{-\alpha}$ is a regular nilpotent element of $\Lg$. Denote the highest root of $\Lg$ by $\theta$, and let $E=E_\theta$. In \cite{FGr}, Frenkel and Gross studied the following connection on the trivial $\LG$-bundle over $\bGm$:
\begin{equation}\label{eq:FG conn}
\nabla_{\FG,\lambda}:=\td+(N+\lambda tE)\frac{\td t}{t},\qquad \lambda\in\bC^\times,
\end{equation}
where $t$ is a coordinate on $\bP^1$. We call the above a Frenkel-Gross connection. In \emph{loc. cit.}, they proved that $\nabla_{\FG,\lambda}$ has regular singularity at $0$ with principal unipotent monodromy, has irregular singularity at $\infty$ with slope $\frac{1}{h}$ where $h$ is the Coxeter number of $\LG$. Moreover, they computed its global monodromy group for all types of $\LG$ and proved that it is cohomologically rigid. 

In \cite{Zhu}, Zhu showed that Frenkel-Gross connection is the same as the de Rham version of the $\ell$-adic Kloosterman sheaf for reductive groups constructed by Heinloth, Ng\^o and Yun in \cite{HNY}, where it is defined as the eigenvalue of certain Hecke eigensheaf; see \S\ref{sss:review on FG} for details.

\subsection{Main results}\label{ss:intro main}
For a $\LG$-connection $\nabla$ on the punctured projective line $U=\bP^1-S$, $S$ a finite subset, we say $\nabla$ is \emph{physically rigid} if for any $\LG$-connection $\nabla'$ on $U$ that is locally isomorphic to $\nabla$ over the punctured disk $D_z^\times$ around each point $z\in S$, there exists a global isomorphism $\nabla\simeq\nabla'$. In this paper, we show that the Frenkel-Gross connection is physically rigid. In fact, we prove the following stronger result:
\begin{thm}\label{t:main}
If $\nabla$ is a $\LG$-connection $\nabla$ on $\bGm$ satisfying:
\begin{enumerate}
	\item $\nabla$ has regular singularity at $0$ with unipotent monodromy,
	\item $\nabla$ has irregular singularity at $\infty$ with slope $\frac{1}{h}$, $h$ the Coxeter number of $\LG$, 
\end{enumerate}
then there exists $\lambda\in\bC^\times$ such that $\nabla\simeq \td+(N+\lambda tE)\frac{\td t}{t}$.
\end{thm}

It is proved in \cite[Theorem 5]{KSformal} that for a formal $\LG$-connection $\nabla$ over a punctured disk $D_\infty^\times$, slope equals $\frac{1}{h}$ is equivalent to that the irregularity of the adjoint connection $\nabla^\Ad$ equals the rank of $\LG$, which is also equivalent to that $\nabla$ is isomorphic to $\nabla_{\FG,\lambda}|_{D_\infty^\times}$ for some $\lambda$. Thus the above theorem is equivalent to the de Rham version of the Conjecture 7.2. in \cite{HNY} when the character $\chi$ at $0$ in the conjecture is trivial. 

Theorem \ref{t:main} immediately implies the physically rigidity of Frenkel-Gross connection:
\begin{cor}\label{c:physically rigid}
The Frenkel-Gross connection $\nabla_{\FG,\lambda}$ is physically rigid.
\end{cor}
\begin{proof}
Let $\nabla_{\FG,\lambda}=\td+(N+\lambda tE)\frac{\td t}{t}$ be a Frenkel-Gross connection, and $\nabla$ is another $\LG$-connection on $\bGm$ satisfying $\nabla|_{D_z^\times}\simeq\nabla_{\FG,\lambda}|_{D_z^\times}$ for $z=0,\infty$. Then from Theorem \ref{t:main}, there exists $\lambda'\in\bC^\times$ such that $\nabla\simeq\td+(N+\lambda' tE)\frac{\td t}{t}$. 

Let $s=t^{-1}$ be a coordinate around $\infty$, then $\nabla_{\FG,\lambda}|_{D_\infty^\times}=\td-(N+\lambda s^{-1}E)\frac{\td s}{s}$. From \cite[Corollary 18]{KSCoxeter}, we know the moduli of formal connections on the punctured disk $D_\infty^\times$ of the form $\td-(N+\lambda s^{-1}E)\frac{\td s}{s}$ is given by $\lambda\in\bC^\times$. Note that in \emph{loc. cit.}, they use $t$ to denote the coordinate on the punctured disk, and use $\lambda$ to denote the coefficient $\td-\lambda(N+s^{-1}E)\frac{\td s}{s}$. After moving $\lambda$ to the coefficient of $E$ by a constant gauge transform, it is no longer up to a $h^{\mathrm{th}}$ root of unity.

Thus $\nabla_{\FG,\lambda}|_{D_\infty^\times}\simeq\nabla|_{D_\infty^\times}$ implies $\lambda=\lambda'$, and further $\nabla_{\FG,\lambda}\simeq\nabla$.
\end{proof}

For $\LG=\GL_n$, the notion of rigidity for connections on vector bundles over a smooth curve is defined and studied in detail by Katz \cite{KatzRigid}. Besides physically rigidity, there is another notion of rigidity called cohomologically rigid, defined as the vanishing of the cohomology of the intermediate extension of the adjoint connection to the whole projective line. For connections on smooth curves with regular singularities, Katz proved that these two notions of rigidity are equivalent \cite[\S1]{KatzRigid}, and Bloch and Esnault proved this for general connections \cite[Remark 4.11]{BE}. 

For a general reductive group $\LG$, one can similarly define these two notions of rigidity, see \cite[\S3]{YunCDM} for a detailed discussion in the language of $\ell$-adic sheaves. Various examples of cohomologically rigid $\LG$-connections on $\bGm$ have been constructed, see \cite{FGr,Chen,KSCoxeter}. However, the relation between cohomologically rigid and physically rigid is unknown for general $\LG$-connections, and there is no known example of a physically rigid $\LG$-connection for a reductive group $\LG$ beyond type A (see \cite[Theorem 7]{KSformal} for the proof of the physically rigidity of Frenkel-Gross connection when $\LG=\SL_n$). To the knowledge of the author, the above Corollary \ref{c:physically rigid} is the first example of physically rigid $\LG$-connections when $\LG$ is not of type A.

\subsection{Outline of the proof of Theorem \ref{t:main}}\label{ss:outline}

Let $\nabla$ be a $\LG$-connection on $\bGm$ that is isomorphic to $\nabla_{\FG,\lambda}$ over $D_0^\times$ and $D_\infty^\times$. The strategy is to identify their corresponding Hecke eigensheaves on the moduli stack of $G$-bundles over $\bGm$ with appropriate level structures at $0$ and $\infty$. This requires the Galois-to-automorphic direction of geometric Langlands correspondence. To explain our strategy, we need to first review some previous results.

\subsubsection{Hecke eigensheaf of Frenkel-Gross connection}\label{sss:review on FG}
Let $F=k(t)$ be the function field of $\bP^1$, $F_x=\bC(\!(t_x)\!)$ the local field at $x\in\bP^1$ with a uniformizer $t_x$, $\cO_x$ the ring of integers in $F_x$. Let $I\subset G(\cO_\infty)$ be the Iwahori subgroup defined as the preimage of a Borel subgroup $B\subset G$ under the evaluation map $G(\cO_\infty)\ra G(\bC)$, $t_\infty\mapsto 0$. Let $I(1)$ be the preimage of the unipotent radical of $B$, and $I(2)=[I(1),I(1)]$, i.e. the second and third terms in the Moy-Prasad filtration of $I$. Let $I^\opp\subset G(\cO_0)$ be the opposite Iwahori subgroup associated to the opposite Borel $B^\opp$. Consider the group scheme $\cG$ on $\bP^1$ such that
\begin{equation}\label{eq:cG}
\cG|_{\bGm}=G\times\bGm,\qquad \cG(\cO_0)=I^\opp,\qquad \cG(\cO_\infty)=I(2).
\end{equation}

Denote the moduli stack of $\cG$-bundles over $\bP^1$ by $\Bun_\cG$. The quotient group $V:=I(1)/I(2)$ acts on $\Bun_\cG$. Moreover, $V$ is an affine space isomorphic to the direct sum of root subspaces of simple affine roots. Let $\phi:V\ra\bGa$ be an additive character that is nonzero when restricted to each root subspace in $V$, and denote by $\cL_\phi$ the pullback via $\phi$ of the exponential D-module on $\bGa$. One of the main results of \cite{HNY}\footnote{They work in the language of $\ell$-adic sheaves.} is that the category of holonomic D-modules on $\Bun_\cG$ that are $(V,\cL_\phi)$-equivariant is isomorphic to the category of vector spaces:
\begin{equation}\label{eq:HNY main}
D_h\!-\!\mathrm{mod}(\Bun_\cG)^{(V,\cL_\phi)}\simeq\mathrm{Vect},
\end{equation}
where the irreducible object is the clean extension of $\cL_\phi$ along a natural open embedding $V\hookrightarrow\Bun_\cG$.

In \cite{Zhu}, Zhu proved that the irreducible $(V,\cL_\phi)$-equivariant holonomic D-module on $\Bun_\cG$ is a Hecke eigensheaf with eigenvalue $\nabla_{\FG,\lambda}$ for some $\lambda\in\bC^\times$ determined by $\phi$. The proof is based on the observation that $\nabla_{\FG,\lambda}$ can be naturally equipped with the structure of an \emph{oper} on $\bGm$, which means it has a reduction to Borel satisfying a transversality condition, see \cite[\S3]{BD} for the precise definition. For connections on a complete smooth curve $X$ that can be equipped with an oper structure, Beilinson and Drinfeld constructed the corresponding Hecke eigensheaves on $\Bun_G(X)$ using a localization functor from Harish-Chander modules of affine Kac-Moody algebra to twisted D-modules on $\Bun_G(X)$. In \cite{Zhu}, Zhu applied a variant of this construction to $\nabla_{\FG,\lambda}$.

\subsubsection{Our strategy}\label{sss:strategy}
Given a connection $\nabla$ that is locally isomorphic to $\nabla_{\FG,\lambda}$, we want to construct an irreducible Hecke eigensheaf of it on $\Bun_\cG$ that is $(V,\cL_\phi)$-equivariant. From \eqref{eq:HNY main}, this must coincide with the Hecke eigensheaf of $\nabla_{\FG,\lambda}$, thus implying the isomorphism $\nabla\simeq\nabla_{\FG,\lambda}$. Our approach is a combination of the constructions in \cite[\S9.6]{FrenkelLecture} and \cite{Zhu}, which consists of a few steps:
\begin{itemize}
\item Step 1: Take a generic oper structure on $\nabla$, i.e. an oper structure over an open subscheme $U\subset\bGm$. The existence of generic oper structure is guaranteed by \cite[Corollary 1.1.]{Arinkin}.

\item Step 2: Choose an appropriate Harish-Chandra module $V_z$ at each singularity $z\in\bP^1-U$. Let $\hg_z$ be the affine Kac-Moody algebra associated to $\fg=\Lie(G)$ at $z$. The restriction to $D_z^\times$ of the generic oper structure fixed in Step 1 defines a local oper over $D_z^\times$, which gives a character of the center of the completed universal enveloping algebra of $\hg_z$ through Feigin-Frenkel isomorphism. We want $V_z$ to be a $\hg_z$-module such that the center acts by this character. Moreover, we want $V_0$ to be $I^\opp$-integrable, $V_\infty$ to be $I(2)$-integrable, and $V_z$ to be $G(\cO_z)$-integrable for each $z\in\bGm-U$.

\item Step 3: Produce a D-module $\cA$ on $\Bun_\cG$ by applying multiple points localization functor to the tensor product of $V_z$, $z\in\bP^1-U$. In order to obtain a D-module on $\Bun_\cG$ that is $(V,\cL_\phi)$-equivariant, we need to quotient $V_\infty$ by the additive character $\phi:V\ra\bC$.

\item Step 4: Show that the D-module $\cA$ on $\Bun_\cG$ constructed in Step 3 is nonzero, holonomic, irreducible, $(V,\cL_\phi)$-equivariant, and is Hecke eigen with eigenvalue $\nabla$.	
\end{itemize}

Then from \eqref{eq:HNY main} we can see $\cA$ must be the same Hecke eigensheaf constructed by Zhu, thus the eigenvalue of $\cA$ is also a Frenkel-Gross connection. Theorem \ref{t:main} follows.

\subsection{Further directions}

\subsubsection{The physically rigidity of Kloosterman sheaves for reductive groups} 
The Frenkel-Gross connection is the characteristic zero analogue of the Kloosterman sheaves for reductive groups constructed in \cite{HNY}. Our proof of physically rigidity uses Beilinson-Drinfeld's construction of Hecke eigensheaves via quantization of Hitchin integrable systems, which exists over an algebraically closed field of characteristic zero. Thus, our method cannot be directly applied to Kloosterman sheaves for reductive groups, which are $\ell$-adic $\LG$-local systems on $\bGm$ over a finite field. However, in \cite{XuZhu}, Xu and Zhu established a bridge between the $\ell$-adic Kloosterman sheaf and Frenkel-Gross connection by constructing the $p$-adic companion of these two local systems, which they call Bessel $F$-isocrystal for reductive groups. It has the same Frobenius eigenvalues as Kloosterman sheaf for reductive groups, and its underlying connection has the same equation as Frenkel-Gross connection. This provides the possibility of using the proof in the current paper to prove the physically rigidity of Kloosterman sheaf for reductive groups through the $p$-adic method. This is the subject of a project in progress with Tsao-Hsien Chen.

\subsubsection{The physically rigidity of other rigid $\LG$-connections} 
Our proof of physically rigidity requires two ingredients: a rigid automorphic datum, which associates a Hecke eigensheaf whose eigenvalue gives the connection; and an alternative construction of the eigensheaf via quantization of Hitchin system. With these two ingredients, it is possible to apply the strategy described in \S\ref{sss:strategy} to other $\LG$-connections. In \cite{KXY}, Kamgarpour, Xu and the author have given these two constructions of Hecke eigensheaves for certain family of $\LG$-connections that generalize hypergeometric sheaves with wild ramification, where $\LG$ is a classical group. Thus one may apply the method in the current paper to prove the physically rigidity of these $\LG$-connections.

\subsubsection{Relation between physically rigidity and cohomologically rigidity}
The theory of rigid automorphic data is developed and studied by Yun in \cite{YunCDM}, based on his joint work with Heinloth and Ng\^o in \cite{HNY}. It provides a method to produce rigid $\LG$-local systems. In \cite[\S1.3.1.]{HNY}, Yun conjectured that weak rigid automorphic data produce cohomologically rigid local systems, while strong automorphic rigid data produce physically rigid local systems. The method in this paper provides a potential way to study the relation between two notions of rigidity for $\LG$-local systems beyond type A and rigid automorphic data.

\subsection{Notations}
Through the paper, we will let $G$ be a simply-connected simple algebraic group over $\bC$ of rank $\ell$. We fix a maximal torus $T$ of $G$ and a Borel subgroup $B$ containing $T$. Let $N$ be the unipotent radical of $B$. Let $B^\opp$ be the opposite Borel. Denote the corresponding sets of simple roots, positive roots and roots by $\Delta\subset\Phi^+\subset\Phi\subset X^*(T)$. Denote the Lie algebras of $G, T, N, B$ by $\fg, \ft, \fn, \fb$. For each root $\alpha\in\Phi$, denote the root subspace by $\fg_\alpha$. Let $\LG$ be the Langlands dual group of $G$, with Lie algebra $\Lg$. The corresponding dual objects are denoted by $\cT, \ct, \cN, \cn, \cB, \cb, \cDelta, \cPhi^+, \cPhi$. 

Let $\bP^1$ be the complex projective line, $t$ a coordinate on it, and $F=k(t)$ the function field. For any $x\in\bP^1$, let $t_x$ be a coordinate around $t_x$, and let $D_x,D_x^\times$ be the formal disk and punctured formal disk. Let $F_x=\bC(\!(t_x)\!)$ and $\cO_x=\bC[\![t_x]\!]$ be the local field and ring of integers.

\textbf{Acknowledgement.} 
The author thanks Alexander Beilinson, Tsao-Hsien Chen and Xinwen Zhu for very valuable discussions. The author particularly thanks Zhu for pointing out Lemma \ref{l:Loc Weyl}, and thanks Chen for suggesting the key reference \cite{FFR}.

\section{The localization functor}
A key observation in Beilinson-Drinfeld's construction of Hecke eigensheaves is that the localization of vacuum module is the sheaf of twisted differential operators on $\Bun_G$. Since we do not know whether a connection that is locally isomorphic to Frenkel-Gross connection has global oper structure, we cannot directly apply their construction. Following \cite[\S9.6]{FrenkelLecture} (see also \cite[\S5]{FrenkelBethe}), we can resolve this issue using Weyl modules, which are generalization of vacuum module that are still $G(\cO)$-integrable. In this section, we compute the localization of Verma modules and Weyl modules in our situation.

\subsection{Recollection on opers and some modules of affine Kac-Moody algebra}
\subsubsection{}\label{sss:hg mods}
Let $\hg$ be the affine Kac-Moody algebra for $\fg$ of critical level, which is a central extension of $\fg(\!(t)\!)$ with central element $\bone$. As vector space, $\hg=\bC\bone+\fg(\!(t)\!)$. The vacuum module is the induced representation
\begin{equation}\label{eq:vac}
\Vac:=\Ind_{\gO+\bC\bone}^{\hg} \bC
\end{equation}
where $\bC$ denotes the trivial represenation of $\gO+\bC\bone$ on which $\bone$ acts as identity.

For any weight $\lambda\in\ft^*$, let $M_\lambda$ be the Verma module of $\fg$ with highest weight $\lambda$. The Verma module of $\hg$ with highest weight $\lambda$ is
\begin{equation}\label{eq:verma}
\bM_\lambda:=\Ind_{\gO+\bC\bone}^{\hg} M_\lambda
\end{equation} 
where $\gO$ acts on $M_\lambda$ through evaluation $t\mapsto 0$, and $\bone$ acts as identify.

For any dominant integral weight $\lambda\in\ft^*$, let $V_\lambda$ be the finite dimensional irreducible representation of $\fg$ with highest weight $\lambda$. The Weyl module of $\hg$ with highest weight $\lambda$ is
\begin{equation}\label{eq:weyl}
\bV_\lambda:=\Ind_{\gO+\bC\bone}^{\hg} V_\lambda
\end{equation}
where $\gO$ acts through $\fg$ and $\bone$ acts as identity.

We will need another module. Let $I(2)\subset G(\!(t)\!)$ be the subgroup of Iwahori as defined in \S\ref{sss:review on FG}. Denote its Lie algebra by $\fI(2)\subset\gO$. We consider the representation induced from the trivial representation of $\fI(2)$:
\begin{equation}\label{eq:vac i2}
\Vac_{\fI(2)}:=\Ind_{\fI(2)+\bC\bone}^{\hg}\bC.
\end{equation}

\subsubsection{}\label{sss:opers}
Let $\Ug$ be the quotient by $\bone-1$ of the completion of the universal enveloping algebra of $\hg$. Its center $\fZ:=Z(\Ug)$, as well as the central support of modules defined in \eqref{eq:vac},\eqref{eq:verma} and \eqref{eq:weyl}, can be described in terms of opers of the dual group $\LG$. 

Let $X$ be either a smooth curve, a formal disk $D=\Spec(\bC[\![t]\!])$, or a punctured disk $D^\times=\Spec(\bC(\!(t)\!))$. An oper is a pair $(\nabla,\cF_B)$ where $\nabla$ is a connection on a $\LG$-bundle $\cF_G$ over $X$, and $\cF_B$ is a Borel reduction of $\cF_G$ satisfying a transversality condition, see \cite[\S3]{BD} or \cite[\S1.1]{FGLocal} for the full definition. Denote the space of opers on $X$ by $\Op_{\Lg}(X)$. Since we assume $\LG$ is of adjoint type, $\Op_{\Lg}(X)$ is a scheme. When $X$ is affine, opers on $X$ can be parametrized by the so-called \emph{oper canonical form}. Here we describe it for $X=D^\times$.

Let $\Lg=\cn^-\oplus\ct\oplus\cn$ be the Cartan decomposition. Denote the half sum of positive coroots by $\crho$. Let $p_{-1}=\sum_{\alpha\in\cDelta}E_{-\alpha}\in\fn^-$ be a regular nilpotent element, and $\{p_{-1},2\crho,p_1\}$ the principal $\Sl_2$-triple where $p_1\in\fn$. Let $\{p_1,p_2,...,p_\ell\}$ be a basis of $\cn^{p_1}$ that are homogeneous with respect to the grading define by $\ad(\check{\rho})$. Let $d_1\leq d_2\leq\cdot\leq d_\ell$ be the fundamental degrees of $\LG$, then $\deg(p_i)=d_i-1$. A $\LG$-oper on $D^\times$ can be uniquely represented by a connection form as follows:
\begin{equation}\label{eq:oper canonical form}
\nabla=\td+(p_{-1}+\sum_{i=1}^\ell v_i(t)p_i)\td t,\quad v_i(t)\in\bC(\!(t)\!).
\end{equation}

Different opers may have the same underlying connection, i.e. their canonical forms are equivalent under gauge transformation.

We consider the following closed subschemes of $\Op_{\Lg}(D^\times)$:
\begin{itemize}
\item $\Op_{\Lg}(D)$: the opers on $D$.
\item $\Op_{\Lg}^{\RS}(D)_{\varpi(\lambda)}$: the opers on $D^\times$ with regular singularity at $0$ whose residue equals to $\lambda\in\ct/\!/W$, $W$ the Weyl group of $\LG$, see \cite[\S2.4]{FGLocal}.
\item $\Op_{\Lg}^{\reg}(D)_{\varpi(\lambda)}$: the opers in $\Op_{\Lg}^{\RS}(D)_{\varpi(\lambda)}$ that can be $\LG(\!(t)\!)$-gauge transformed to the trivial connection, see \cite[\S2.9]{FGLocal}.
\item $\Op_{\Lg}(D)_{\leq 1/h}$: the opers on $D^\times$ whose slope is smaller or equal to $1/h$, $h=d_\ell$ the Coxeter number of $\LG$; see \cite[\S2.7]{CK} for the notion of slope of opers.
\end{itemize}

\subsubsection{}\label{sss: support of center via opers}
The Feigin-Frenkel isomorphism \cite{FF} gives
\begin{equation}\label{eq:FF}
\fZ\simeq\Fun\Op_{\Lg}(D^\times).
\end{equation}

Thus we can describe the support of the actions of $\fZ$ on the $\hg$-modules defined in \S\ref{sss:hg mods} using the closed subspaces of $\Op_{\Lg}(D^\times)$ defined in \S\ref{sss:opers}. Moreover, note from \eqref{eq:oper canonical form} that $\Op_{\Lg}(D^\times)$ is an affine space. Denote $v_i(t)=\sum_j v_{ij}t^{-j-1}\in\bC(\!(t)\!)$, then $v_{ij}\in\Fun\Op_{\Lg}(D^\times)$ maps to a set of generators $S_{ij}$ of $\fZ$ under \eqref{eq:FF}, called Segal-Sugawara operators. See \cite{MolevBook} for an introduction of Segal-Sugawara operators.

Let $\fz$ be the center of $\Vac$ as the affine Vertex algebra. By the Feigin-Frenkel isomorphism over formal disk, we also have
\begin{equation}\label{eq:FF unr}
\Fun\Op_{\Lg}(D)\simeq\fz=\im(\fZ\ra\End_{\hg}\Vac)=\End_{\hg}\Vac.
\end{equation}

Let $\fZ^{\RS}_{\varpi(-\lambda-\rho)}:=\im(\fZ\ra\End_{\hg}\bM_\lambda)$, where $\rho$ is the half sum of positive roots. By \cite[Corollary 13.3.2]{FGLocal}, we have
\begin{equation}\label{eq:FF RS}
\Fun\Op_{\Lg}^{\RS}(D)_{\varpi(-\lambda-\rho)}\simeq\fZ^{\RS}_{\varpi(-\lambda-\rho)}\simeq\End_{\hg}\bM_\lambda.
\end{equation}

Let $\fZ^{\reg}_{\varpi(-\lambda-\rho)}:=\im(\fZ\ra\End_{\hg}\bV_\lambda)$. By \cite[Theorem 1]{FGWeyl}, we have
\begin{equation}\label{eq:FF reg}
\Fun\Op_{\Lg}^{\reg}(D)_{\varpi(-\lambda-\rho)}\simeq\fZ^{\reg}_{\varpi(-\lambda-\rho)}=\End_{\hg}\bV_\lambda.
\end{equation}

Let $\fZ_{\fI(2)}:=\im(\fZ\ra\End_{\hg}\Vac_{\fI(2)})$. By \cite[\S5.2]{Zhu}, we have
\begin{equation}\label{eq:FF i(2)}
\Fun\Op_{\Lg}(D)_{\leq 1/h}\simeq\fZ_{\fI(2)}.
\end{equation}

\subsection{Construction of D-modules via localization functor}\label{ss:Loc}
Let $\cG$ be a smooth affine group scheme over $\bP^1$ such that over some open dense subset $U\subset\bP^1$, $\cG|_U=G\times U$; at each point $x\in \bP^1-U$, $\cG(\cO_x)$ is either a parahoric group, or a pro-unipotent subgroup of a parahoric subgroup that contains a Moy-Prasad subgroup of it. Morever, we assume the moduli stack $\Bun_\cG$ of $\cG$-bundles on $\bP^1$ is good in the sense of Beilinson-Drinfeld \cite[\S1.1.1]{BD}. We also assume the canonical bundle $\omega_{\Bun_\cG}$ on $\Bun_{\cG}$ admits a square root line bundle $\omega_{\Bun_\cG}^{1/2}$. All of these assumptions are satisfied by the group scheme $\cG$ defined in \eqref{eq:cG}.

\subsubsection{}
For each $x\in\bP^1$, let $\Bun_{\cG,x}$ be the moduli stack of $\cG$-bundles with a trivialization over $D_x$. Thus $\pi:\Bun_{\cG,x}\ra\Bun_\cG$ is a $\cG(\cO_x)$-bundle. Denote $K_x=\cG(\cO_x)$, then the pullback of $\omega_{\Bun_\cG}^{1/2}$ to $\Bun_{\cG,x}$ admits a $(\hg_x,K_x)$-actions. Denote by $\cD_{\Bun_\cG}'$ (resp. $\cD_{\Bun_{\cG,x}}'$) the sheaf of differential operators on $\Bun_\cG$ twisted by $\omega_{\Bun_\cG}^{1/2}$ (resp. on $\Bun_{\cG,x}$ twisted by the pullback of $\omega_{\Bun_\cG}^{1/2}$). For any $(\hg_x,K_x)$-module $M$, we associate a twisted D-module on $\Bun_{\cG,x}$ and $\Bun_\cG$ by the standard localization construction:
\begin{equation}\label{eq:Loc}
\tLoc(M):=\cD_{\Bun_{\cG,x}}'\otimes_{U(\hg_x)}M,\quad \Loc(M):=(\pi_*(\tLoc(M)))^{K_x}.
\end{equation} 
For generality on the localization construction, see \cite[\S17.2]{FB} for a detailed introduction.

It is well-known that for $x\in U$ and $M=\Vac$, $\Loc(\Vac)=\cD_{\Bun_\cG}$.

\subsubsection{}
Assume for some $x\in\bP^1-U$, $\cG(\cO_x)=I$ is an Iwahori subgroup. We compute $\Loc(\bM_\lambda)$ for an integral weight $\lambda$. Denote by $\cG'$ the group scheme that is the same as $\cG$ on $\bP^1-\{x\}$ and $\cG'(\cO_x)=I(1)$ is the pro-unipotent radical of $I$. Thus $\pi_1:\Bun_{\cG,x}\ra\Bun_{\cG'}$ is a $I(1)$-bundle, and $\pi_2:\Bun_{\cG'}\ra\Bun_\cG$ is a $T$-bundle. Since $\lambda$ is integral and $G$ is simply connected, it gives rise to a character $\lambda:T\ra\bGm$ which we still denote by $\lambda$. Consider the associated line bundle
\begin{equation}\label{eq:cL_lambda}
\cL_\lambda=\Bun_{\cG'}\times^{T,\lambda}\bGm
\end{equation}
over $\Bun_\cG$. Following the discussion in \cite[\S9.5]{FrenkelRam}, we have

\begin{lem}\label{l:Loc Verma}
For an integral weight $\lambda$ and $x\in\bP^1-U$ where $\cG(\cO_x)=I$, $\Loc(\bM_\lambda)\simeq\cD_{\Bun_\cG}'\otimes_{\cO_{\Bun_\cG}}\cL_\lambda$.
\end{lem}
\begin{proof}
Since $\bM_\lambda\simeq U(\hg)/U(\hg)\fI(1)\otimes_{U\ft}\bC_\lambda$ and $\pi=\pi_2\circ\pi_1$, we have
$$
\Loc(\bM_\lambda)=(\pi_{2*}(\cD_{\Bun_{\cG'}}'\otimes_{U\ft}\bC_\lambda))^T.
$$
On any neighborhood $\cY$ in $\Bun_\cG$ over which the $T$-bundle $\Bun_{\cG'}\ra\Bun_\cG$ can be trivialized, 
$$
\Loc(\bM_\lambda)|_{\cY}\simeq\cD_{\cY}'\otimes_{\bC}(\pi_{2*}(\cD_T\otimes_{U\ft}\bC_\lambda))^T\simeq\cD_{\cY}'\otimes_{\bC}\bC_\lambda.
$$ 
The lemma follows from gluing the above isomorphism together.
\end{proof}

\subsubsection{}
Next consider any $x\in U$, where $\cG(\cO_x)=G(\cO_x)$. We compute $\Loc(\bV_\lambda)$ for a dominant integral weight $\lambda$ in a similar way. Denote by $\Bun_{\cG,\bar{x}}$ the moduli space of $\cG$-bundles with a trivilization of its fiber at $x$. Then $\pi_1:\Bun_{\cG,x}\ra\Bun_{\cG,\bar{x}}$ is a $G(t_x\cO_x)$-bundle where $G(t_x\cO_x)=\ker(G(\cO_x)\ra G(\bC))$, and $\pi_2:\Bun_{\cG,\bar{x}}\ra\Bun_\cG$ is a $G$-bundle. Consider associated vector bundle
\begin{equation}\label{eq:V_lambda}
\cV_\lambda=\Bun_{\cG,\bar{x}}\times^G V_\lambda
\end{equation}
on $\Bun_\cG$. We have
\begin{lem}\label{l:Loc Weyl}
For a dominant integral weight $\lambda$ and $x\in U$ where $\cG(\cO_x)=G(\cO_x)$, $\Loc(\bV_\lambda)\simeq\cD_{\Bun_\cG}'\otimes_{\cO_{\Bun_\cG}}\cV_\lambda$. 
\end{lem}
\begin{proof}
Since $\bV_\lambda\simeq U(\hg)/U(\hg)t_x\fg[\![t_x]\!]\otimes_{U\fg}V_\lambda$ and $\pi=\pi_2\circ\pi_1$, we have
$$
\Loc(\bV_\lambda)=(\pi_{2*}(\cD_{\Bun_{\cG,\bar{x}}}'\otimes_{U\fg}V_\lambda))^G.
$$
On any neighborhood $\cY$ in $\Bun_\cG$ over which the $G$-bundle $\Bun_{\cG,\bar{x}}\ra\Bun_\cG$ can be trivialized,
\begin{align*}
\Loc(\bV_\lambda)|_{\cY}
&\simeq\cD_{\cY}'\otimes_{\bC}(\pi_{2*}(\cD_G\otimes_{U\fg}V_\lambda))^G
\simeq\cD_{\cY}'\otimes_{\bC}(\cO_G\otimes_{\bC}V_\lambda)^G\\
&\simeq \cD_{\cY}'\otimes_{\bC}\mathrm{Hom}_G(V_\lambda^*,\cO_G)
\simeq\cD_{\cY}'\otimes_{\bC}V_\lambda.
\end{align*}
The lemma follows from gluing the above isomorphism together.
\end{proof}

Later we will only use that $\Loc(\bM_\lambda)$ (resp. $\Loc(\bV_\lambda)$) is locally isomorphic to $\cD'\otimes_{\bC}\bC_\lambda$ (resp. $\cD'\otimes_{\bC}V_\lambda$).

\section{A flatness property}\label{s:flatness}
Let $\cG$ be the group scheme defined in \eqref{eq:cG}, i.e. $\cG(\cO_0)=I^\opp$, $\cG(\cO_\infty)=I(2)$, and $\cG|_{\bGm}=G\times\bGm$. From \cite[Lemma 17, Remark 6.1]{Zhu}, we know it satisfies all the assumptions in \S\ref{ss:Loc}. Let $z_1,z_2,...,z_k$ be a finite subset of $\bGm$. Denote $U=\bGm-\{z_1,...,z_k\}$, $S=\{0,z_1,...,z_k,\infty\}$. For $z=0$, let $\lambda_0$ be an integral weight; for each $z_i$, let $\lambda_i$ be a dominant integral weight. 

Denote $\hg_*:=(\oplus_{z\in S}\hg_z)/\langle\bone_z-\bone_0\rangle$ and $K_*=\prod_{z\in S}\cG(\cO_z)$. The multiple points version of the localization construction produces a twisted D-module on $\Bun_\cG$ from a $(\hg_*,K_*)$-module. From Lemma \ref{l:Loc Verma}, Lemma \ref{l:Loc Weyl} and that $\Loc(\Vac_{\fI(2)})=\cD_{\Bun_\cG}'$ (see \cite[\S3.1]{Zhu} for the last equality), we have
\begin{equation}\label{eq:global loc}
\cD'_{z_i,\lambda_i}:=\Loc(\bM_{\lambda_0}^\opp\otimes_i\bV_{\lambda_i}\otimes\Vac_{\fI(2)})
\simeq\cD_{\Bun_\cG}'\otimes\cL_{\lambda_0}\otimes_i\cV_{\lambda_i}.
\end{equation}

In the above, $\bM_{\lambda_0}^\opp$ is the Verma module for the opposite Borel $B^\opp$, and $\Vac_{\fI(2)}$ defined in \eqref{eq:vac i2} is the vacuum module for $(\hg_\infty,I(2))$. In the following, we define an algebra $A$ over which $\cD'_{z_i,\lambda_i}$ is flat.

\subsection{Flatness over an algebra}\label{ss:flat over A}
The localization functor $\Loc$ is functorial, sending an endomorphism of a $(\hg_x,K_x)$-module $V$ to an endomorphism of the resulting twisted D-module on $\Bun_\cG$. Thus we get a morphism
$$
\fZ\ra\End_{\hg_x}(V)\ra\End_{\cD'}(\Loc(V)).
$$ 

At $z=0$, let $w_0\in G(\bC)$ be a representative of the longest element of Weyl group. Consider a new $\hg_0$-module structure on $\bM_{w_0\lambda_0}$ by $X\cdot v=\Ad_{w_0}(X)v$, then $\bM_{\lambda_0}^\opp$ is isomorphic to $\bM_{w_0\lambda_0}$ as $\hg_0$-modules, thus have the same central support in $\fZ$. From \eqref{eq:FF RS}, we get an action of $\Fun\Op_{\Lg}^\RS(D_0)_{\varpi(-w_0\lambda_0-\rho)}$ on $\cD'_{z_i,\lambda_i}$.

At $z=z_i$, $i=1,...,k$, we get from \eqref{eq:FF reg} an action of $\Fun\Op_{\Lg}^\reg(D_{z_i})_{\varpi(-\lambda_i-\rho)}$ on $\cD'_{z_i,\lambda_i}$.

At $z\in U$, since $\Loc(\bM_{\lambda_0}^\opp\otimes_i\bV_{\lambda_i}\otimes\Vac_{\fI(2)})$ would not change if we add the point $z$ and tensor with vacuum module $\Vac_z$, we get from \eqref{eq:FF unr} an action of $\Fun\Op_{\Lg}(D_z)$ on $\cD'_{z_i,\lambda_i}$. 

At $z=\infty$, we get from \eqref{eq:FF i(2)} an action of $\Fun\Op_{\Lg}(D_\infty)_{\leq 1/h}$ on $\cD'_{z_i,\lambda_i}$. Moreover, denote $V=I(1)/I(2)\simeq\fI(1)/\fI(2)$, which is acted on by $T$. We see from \cite[Proposition 15]{Zhu}\footnote{This proposition can be proved alternatively using Segal-Sugawara operators, see \cite[Proposition 28.(i)]{KXY}.} that both $U(V)$ and $\fZ_{\fI(2)}=\Fun\Op_{\Lg}(D_\infty)_{\leq 1/h}$ act on $\Vac_{\fI(2)}$, and their intersection inside $\End\Vac_{\fI(2)}$ is $U(V)^T$. Thus we also get an action of $U(V)$ on $\cD'_{z_i,\lambda_i}$ that is compatible with the action of $\Fun\Op_{\Lg}(D_\infty)_{\leq 1/h}$ over $U(V)^T$.

Now consider the following space of global opers:
\begin{equation}\label{eq:Op_*}
\begin{split}
\Op_*:=&\Op_{\Lg}(U)\times_{\Op_{\Lg}(D_0^\times)}\Op_{\Lg}^\RS(D_0)_{\varpi(-w_0\lambda_0-\rho)}\times_{i,\Op_{\Lg}(D_{z_i}^\times)}\Op_{\Lg}^\reg(D_{z_i})_{\varpi(-\lambda_i-\rho)}\\
&\times_{\Op_{\Lg}(D_\infty^\times)}\Op_{\Lg}(D_\infty)_{\leq 1/h}.
\end{split}
\end{equation}

We can assemble the above actions of quotients of $\fZ$ at each point $x\in\bP^1$ together using the technique of conformal blocks as in \cite[\S2.4]{Zhu}. From the above discussion, we get an action of $\Fun\Op_*$ on $\cD'_{z_i,\lambda_i}$. Since the action of $\Fun\Op_{\Lg}(D_\infty)_{\leq 1/h}\simeq\fZ_{\fI(2)}\supset U(V)^T$ extends to $U(V)$, we get further the action of
\begin{equation}\label{eq:A}
A:=\Fun\Op_*\otimes_{U(V)^T}U(V)
\end{equation}
on $\cD_{z_i,\lambda_i}'$.

The following is a key ingredient to the proof of Theorem \ref{t:main}:
\begin{prop}\label{p:flatness}
$\cD_{z_i,\lambda_i}'$ is flat over $A$.
\end{prop}

We prove this by proving the flatness of the associated graded module with respect to the natural filtration on the sheaf of differential operators and the PBW filtration on the center $\fZ$.

\subsection{The associated graded of $\cD_{z_i,\lambda_i}'$}\label{ss:grD_z,lambda}
We first describe the associated graded of $\cD_{z_i,\lambda_i}'$. For each vector bundle $\cV_{\lambda_i}$ on $\Bun_\cG$ defined by \eqref{eq:V_lambda}, since $\cD_{\Bun_\cG}'\otimes\cV_{\lambda_i}$ is locally copies of $\cD_{\Bun_\cG}'$, its associated graded is a vector bundle $\cV_{\lambda_i}'$ on the cotangent stack $T^*\Bun_\cG$ with fiber $V_{\lambda_i}$, which can be described canonically as follows. 

Recall $\cD_{\Bun_\cG}'\otimes\cV_{\lambda_i}$ is the descent of $\cD_{\Bun_{\cG,\bar{z}_i}}'\otimes_{U\fg}V_{\lambda_i}$. Its associated graded is the $G$-equivariant module $\cO_{T^*\Bun_{\cG,\bar{z}_i}}\otimes_{S\fg}V_{\lambda_i}$ on $T^*\Bun_{\cG,\bar{z}_i}$, where the map $S\fg=\Fun\fg^*\ra\cO_{T^*\Bun_{\cG,\bar{z}_i}}$ comes from the moment map $\mu:T^*\Bun_{\cG,\bar{z}_i}\ra\fg^*$, and the action of $S\fg$ on $V_{\lambda_i}$ factors through the quotient by degree greater than one terms. From Hamiltonian reduction, we have $T^*\Bun_\cG\simeq\mu^{-1}(0)/G$. Denote $p:\mu^{-1}(0)\ra T^*\Bun_\cG$, then
\begin{equation}\label{eq:gr cV_lambda}
\cV_{\lambda_i}'=p_*((\cO_{T^*\Bun_{\cG,\bar{z}_i}}\otimes_{S\fg}V_{\lambda_i})\big|_{\mu^{-1}(0)})^G.
\end{equation}

Note that for $\cL_{\lambda_0}$, the associated graded of $\bC_{\lambda_0}$ becomes the trivial module over $\ft$. Putting the above together, we get
\begin{equation}\label{eq:grD'}
\gr\cD_{z_i,\lambda_i}'\simeq\otimes_i\cV_{\lambda_i}'
\end{equation}
where we omit the pushforward from $T^*\Bun_\cG$ to $\Bun_\cG$ for simplicity of notation.

\subsection{The associated graded of local opers}\label{ss:gr local opers}
Next we describe the associated graded of the ring of functions on various subspaces of local opers defined in \S\ref{sss:opers} in terms of Hitchin base.

\subsubsection{}
Let $X$ be either a smooth curve, a formal disk $D$ or a punctured formal disk $D^\times$. We fix a set of homogeneous invariant polynomials on $\fg$ as in \cite[\S4.1]{Zhu} using Kostant section. The Hitchin base over $X$ is
\begin{equation}
\Hit(X):=\Gamma(X,\fg^*\dl G\times^{\bGm}\omega_X^\times)\simeq\bigoplus_{i=1}^\ell\omega_X^{d_i},
\end{equation}
where $\omega$ is the canonical bundle on $X$. Note that the above isomorphism is non-canonical, depending on our choice of homogeneous invariant polynomials generating $\bC[\fg]^G$. It can be regarded as the classical limit of the space of opers (see \cite[\S3.1.12]{BD}):
\begin{equation}\label{eq:gr global op=Hit}
\gr\Fun\Op_{\Lg}(X)\simeq\Fun\Hit(X).
\end{equation}

For $X=D^\times, D$, we know from \cite[Theorem 2.5.2, Theorem 3.7.8]{BD} that
\begin{equation}\label{eq:gr fZ, fz}
\gr\fZ\simeq\gr\Fun\Op_{\Lg}(D^\times)\simeq\Fun\Hit(D^\times),\quad
\gr\fz\simeq\gr\Fun\Op_{\Lg}(D)\simeq\Fun\Hit(D).
\end{equation}

Also, let $\Op_{\Lg}^\RS(D)$ be the space of opers on $D^\times$ with regular singularity, and $\Hit^\RS(D)=\bigoplus_{i=1}^\ell\omega_D^{d_i}(d_i)$ the Hitchin base with regular singularity, then
\begin{equation}\label{eq:gr op RS}
\gr\Fun\Op_{\Lg}^\RS(D)\simeq\Fun\Hit^\RS(D).
\end{equation}

If we fix a coordinate $t$ on $D^\times$, then we can write a section in $\Hit(D^\times)$ as $\sum_{i=1}^{\ell}h_i(t)(\td t)^{d_i}$, where $h_i(t)=\sum_j h_{ij}t^{-j-1}\in\bC(\!(t)\!)$. Thus
\begin{equation}\label{eq:h_ij}
\Fun\Hit(D^\times)=\bC[h_{ij}\ |\ 1\leq i\leq \ell, j\in\bZ].
\end{equation}
The elements $h_{ij}\in\Fun\Hit(D^\times)$ are the symbols of the Segal-Sugawara operators $S_{ij}\in\fZ$ under $\gr\fZ\simeq\Fun\Hit(D^\times)$ defined in \S\ref{sss: support of center via opers}, thus also correspond to the symbols of coefficeints $v_{ij}$ of the oper canonical form, i.e. $h_{ij}=\bar{v}_{ij}=\bar{S}_{ij}$.

\subsubsection{}
Next we compute $\gr\Fun\Op_{\Lg}^\RS(D)_{\varpi(\lambda)}$ for any $\lambda\in\ct$. Recall the quotient map $\varpi:\Lg\ra\Lg\dl\LG$ is isomorphic on the Kostant section $p_{-1}+\cn^{p_1}$. From \cite[Lemma 2.4.2]{FGLocal}, we have
$$
\Op_{\Lg}^\RS(D)_{\varpi(\lambda)}=\{\td+(p_{-1}+\sum_{i=1}^{\ell}t^{-d_i}(v_{i,d_i-1}+t\bC[\![t]\!])p_i)\td t\ |\ \varpi(p_1+\sum_{i=1}^\ell v_{i,d_i-1}+\frac{p_1}{4})=\varpi(\lambda)\}.
$$
Note that our $d_i$ correspond to $d_i+1$ in the \emph{loc. cit.}. Following the notations in \cite[\S4.3]{Zhu}, we denote
\begin{equation}\label{eq:Hit_I}
\Hit(D)_{\fI}:=\bigoplus_{i=1}^\ell t^{-d_i+1}\bC[\![t]\!](\td t)^{d_i}\subset\Hit(D^\times).
\end{equation}
Then it is clear from $h_{ij}=\bar{v}_{ij}$ and the above discussion that
\begin{equation}\label{eq:gr oper RS lambda}
\gr\Fun\Op_{\Lg}^\RS(D)_{\varpi(\lambda)}\simeq\Fun\Hit(D)_{\fI}=\bC[h_{ij}\ |\ j\leq d_i-2, 1\leq i\leq \ell].
\end{equation}

\subsubsection{}
The algebra $\gr\Fun\Op_{\Lg}(D)_{\leq 1/h}$ has been computed in \cite{Zhu}, where $\Op_{\Lg}(D)_{\leq 1/h}$ is denoted by $\Op_{\Lg}(D)_{\fI(2)}$. Denote
\begin{equation}\label{eq:Hit_I(2)}
\Hit(D)_{\fI(2)}=\bigoplus_{i=1}^{\ell-1}t^{-d_i}\bC[\![t]\!](\td t)^{d_i}\oplus t^{-d_\ell-1}\bC[\![t]\!](\td t)^{d_\ell}.
\end{equation}
From Proposition 13 and Lemma 16 of \emph{loc. cit.}, we get
\begin{equation}\label{eq:gr op 1/h}
\gr\Fun\Op_{\Lg}(D)_{\leq 1/h}\simeq\gr\Fun\Hit(D)_{\fI(2)}=\bC[h_{ij}\ |\ j\leq d_i-1, 1\leq i\leq \ell-1; j\leq d_\ell, i=\ell].
\end{equation}

\subsubsection{}
The algebra $\gr\Fun\Op_{\Lg}^\reg(D)_{\varpi(-\lambda-\rho)}$,
where $\lambda$ is a dominant integral weight, has been computed in \cite[\S2.2]{FFR}. For simplicity of notations, we denote $\cZ_\lambda:=\Fun\Op_{\Lg}^\reg(D)_{\varpi(-\lambda-\rho)}$ and $\bcZ_\lambda=\gr\cZ_\lambda$. Thus $\cZ_0=\Fun\Op_{\Lg}(D)\simeq\fz$, $\bcZ_0=\gr\fz=(\Fun\gO)^{\gO}$. From \emph{loc. cit.}, we have an embedding
$$
\bcZ_\lambda\hookrightarrow(\End_{\bC}(V_\lambda)\otimes\Fun\gO)^{\gO}
$$
which is only an isomorphism when $\lambda$ is zero or minuscule. There is a natural embedding of $\bcZ_0=(\Fun\gO)^{\gO}$ into $(\End_{\bC}(V_\lambda)\otimes\Fun\gO)^{\gO}$. Moreover, Lemma 2 of the \emph{loc. cit.} says that $\bcZ_0\subset\bcZ_\lambda$, and $\bcZ_\lambda$ is a free module over $\bcZ_0$ with a basis that can be described as follows.

For any $\mu\in\fg^*$, one can assign a commutative subalgebra $\cA_\mu\subset U\fg$ called the quantum shift of argument subalgebra, see \cite[\S2.7]{FFT} for the precise definition. Let $\mu$ by the image of the regular nilpotent element $f=p_{-1}$ under isomorphism $\fg\simeq\fg^*$ given by Killing form, then we have a particular subalgebra $\cA_f\subset U\fg$.
Let $\pi_\lambda:U\fg\ra\End_{\bC}V_\lambda$ be the finite dimensional irreducible representation. Denote \begin{equation}\label{eq:N_lambda}
N_\lambda:=\pi_\lambda(\cA_f)\subset\End_{\bC}V_\lambda.
\end{equation}

We know from \cite[Corollary 2]{FFR10} that $V_\lambda$ is cyclic as $\cA_f$-module. Thus $V_\lambda$ is rank one free $N_\lambda$-module. From the discussion in \cite[\S2.2]{FFR}, $N_\lambda$ can be lifted to a subspace in $\bcZ_\lambda$, which we still denote by $N_\lambda$. Then we know from Lemma 4 of the \emph{loc. cit.} that $\bcZ_\lambda$ is a free $\bcZ_0$-module generated by $N_\lambda$:
\begin{equation}\label{eq:bZ_lambda=bZ_0 N_lambda}
\bcZ_\lambda\simeq\bcZ_0\otimes_{\bC}N_\lambda.
\end{equation}

\subsection{The associated graded of $\Fun\Op_*$}
Recall from the discussion in \S\ref{ss:flat over A} that $U(V)^T$ maps to $\Fun\Op_*$, thus $\gr\Fun\Op_*$ is a $\Fun(V^*\dl T)$-module. For each dominant integral weight $\lambda_i$ attached to $z_i$, $i=1,...,k$, we associate a vector space $N_{\lambda_i}$ defined as in \eqref{eq:N_lambda}.
\begin{lem}\label{l:grFunOp*}
$\gr\Fun\Op_*$ is a free $\Fun(V^*\dl T)$-module. Moreover, we have a non-canonical isomorphism of $\Fun(V^*\dl T)$-modules:
\begin{equation}\label{eq:grFunOp*}
\gr\Fun\Op_*\simeq \Fun(V^*\dl T)\otimes_{\bC,i}N_{\lambda_i}.
\end{equation}
\end{lem}
\begin{proof}
For simplicity of notations we assume there is only one $z_i$, i.e. $U=\bGm-\{z\}$, where at $z$ we attach a dominant integral weight $\lambda$. It will be clear from the proof that the generalization to multiple $z_i$ is completely formal.

From \eqref{eq:Op_*}, \eqref{eq:gr global op=Hit}, \eqref{eq:gr oper RS lambda}, \eqref{eq:gr op 1/h} and \eqref{eq:bZ_lambda=bZ_0 N_lambda}, we have
\begin{equation}
\begin{split}
\gr\Fun\Op_*\simeq&\Fun\Hit(\bGm-\{z\})\otimes_{\Fun\Hit(D_0^\times)}\Fun\Hit(D_0)_{\fI}\otimes_{\Fun\Hit(D_\infty^\times)}\Fun\Hit(D_\infty)_{\fI(2)}\\
&\otimes_{\Fun\Hit(D_z^\times)}\bcZ_\lambda.
\end{split}
\end{equation}

Following the notations in \cite[\S6]{Zhu}, we define
$$
\Hit_{\cG}^\RS(\bP^1-\{z\}):=\Hit(\bGm-\{z\})\times_{\Hit(D_z^\times)}\Hit^\RS(D_z)\times_{\Hit(D_0^\times)}\Hit(D_0)_{\fI}\times_{\Hit(D_\infty^\times)}\Hit(D_\infty)_{\fI(2)}.
$$
Note that $\Op_{\Lg}^\reg(D_z)_{\varpi(-\lambda-\rho)}\subset\Op_{\Lg}^\RS(D_z)$ implies $\Spec\bcZ_\lambda\subset\Hit^\RS(D_z)$. Thus we can write
$$
\gr\Fun\Op_*\simeq\Fun\Hit_{\cG}^\RS(\bP^1-\{z\})\otimes_{\Fun\Hit^\RS(D_z)}\bcZ_\lambda.
$$
We study the actions of $\Fun\Hit^\RS(D_z)$ on $\Fun\Hit_{\cG}^\RS(\bP^1-\{z\})$ and $\bcZ_\lambda$ respectively.

First, consider the action on $\bcZ_\lambda$. In terms of the generators $h_{ij}$ in \eqref{eq:h_ij}, we have
$$
\Fun\Hit^\RS(D_z)=\Fun\Hit(D_z)\otimes_{\bC}M,\quad M:=\bC[h_{ij}\ |\ 1\leq i\leq \ell, 0\leq j\leq d_i].
$$
Recall $\bcZ_0=\gr\fz\simeq\Fun\Hit(D_z)$. From the discussion in \cite[\S2.2]{FFR}, $N_\lambda\subset\bcZ_\lambda$ is mapped onto by $M$. Denote $I_\lambda:=\ker(M\ra N_\lambda)$, then from \eqref{eq:bZ_lambda=bZ_0 N_lambda} we have
\begin{equation}\label{eq:act on bcZ_lambda}
\bcZ_\lambda\simeq\Fun\Hit(D_z)\otimes M/I_\lambda\simeq\Fun\Hit^\RS(D_z)/\Fun\Hit(D_z)I_\lambda.
\end{equation}

Next, we study the action on $\Fun\Hit_{\cG}^\RS(\bP^1-\{z\})$. Consider
\begin{equation}\label{eq:Hit'}
\begin{split}
\Hit_{\cG}^\RS(\bP^1-\{z\})':&=\Hit_{\cG}^\RS(\bP^1-\{z\})\times_{\Hit(D_\infty^\times)}\Hit^\RS(D_\infty)\\
&=\bigoplus_{i=1}^\ell\Gamma(\bP^1,\omega_{\bP^1}^{d_i}((d_i-1)\cdot 0+d_i\cdot z+d_i\cdot\infty)).
\end{split}
\end{equation}
Note from \eqref{eq:Hit_I(2)} that $\Hit^\RS(D_z)\subset\Hit(D_z)_{\fI}$, thus $\Hit_{\cG}^\RS(\bP^1-\{z\})'\subset\Hit_{\cG}^\RS(\bP^1-\{z\})$.

Claim: the following composition is an isomorphism:
\begin{equation}\label{eq:iota}
\iota:\Hit_{\cG}^\RS(\bP^1-\{z\})'\hookrightarrow\Hit_{\cG}^\RS(\bP^1-\{z\})\xrightarrow{\mathrm{res}|_{D_z^\times}}\Hit^\RS(D_z)\twoheadrightarrow\Hit^\RS(D_z)/\Hit(D_z)\simeq\Spec M.
\end{equation}
In fact, observe that $\omega_{\bP^1}^{d_i}((d_i-1)\cdot 0+k\cdot z+d_i\cdot\infty)\simeq\cO_{\bP^1}(k-1)$. Then it is clear from \eqref{eq:Hit'} that for any $1\leq i\leq \ell$ and $1\leq k\leq d_i$, the meromorphism $d_i$-form in $\Hit_{\cG}^\RS(\bP^1-\{z\})'$ that has a pole of order $k$ at $z$ is unique up to a nonzero scalar and such forms with a pole of order smaller than $k$ at $z$. These forms give a basis of $\Hit_{\cG}^\RS(\bP^1-\{z\})'$, which maps to a basis of $\Spec M$. This gives the claimed isomorphism.

The kernel of the above map $\Hit_{\cG}^\RS(\bP^1-\{z\})\ra\Spec M$ is
$$
\Hit_{\cG}(\bP^1):=\Hit_{\cG}^\RS(\bP^1-\{z\})\times_{\Hit(D_z^\times)}\Hit(D_z)=\Hit(\bGm)\times_{\Hit(D_0^\times)}\Hit(D_0)_{\fI}\times_{\Hit(D_\infty^\times)}\Hit(D_\infty)_{\fI(2)}.
$$

Thus we have an isomorphism
$$
\Hit_{\cG}^\RS(\bP^1-\{z\})=\Hit_{\cG}(\bP^1)\oplus\Hit_{\cG}^\RS(\bP^1-\{z\})'.
$$
Note that the embedding $\Hit_{\cG}^\RS(\bP^1-\{z\})\ra\Hit^\RS(D_z)\simeq\Hit(D_z)\times\Spec M$ is not a direct sum of maps to $\Hit(D_z)$ and $\Spec M$. This is because $\Spec M\subset\Hit^\RS(D_z)$ is the space of differential forms with zero unramified part, but the Laurent expansions at $z$ of nonzero forms in $\Hit_{\cG}^\RS(\bP^1-\{z\})'$ would have nonzero regular part, thus are not mapped into $\Spec M\subset\Hit^\RS(D_z)$. We can fix this using the isomorphism \eqref{eq:iota}. Consider morphism
\begin{align*}
\psi:\Hit^\RS(D_z)=&\Hit(D_z)\times\Spec M
\xrightarrow{\mathrm{id}\times\iota^{-1}}\Hit(D_z)\times\Hit_{\cG}^\RS(\bP^1-\{z\})'\\
&\xrightarrow{\mathrm{id}\times\mathrm{res}|_{D_z^\times}}\Hit(D_z)\times\Hit^\RS(D_z)
\xrightarrow{\mathrm{addition}}\Hit^\RS(D_z).
\end{align*}
which is a $\Hit(D_z)$-equivariant automorphism. Then $\Hit_{\cG}^\RS(\bP^1-\{z\})\ra\Hit^\RS(D_z)\xrightarrow{\psi^{-1}}\Hit^\RS(D_z)$ is the direct sum of $\Hit_\cG(\bP^1)\ra\Hit(D_z)$ and $\mathrm{id}_M$.

Putting the above discussions together, we get:
$$
\gr\Fun\Op_*\simeq(\Fun\Hit_\cG(\bP^1)\otimes M)\otimes_{\Fun\Hit(D_z)\otimes M}\Fun\Hit(D_z)\otimes M/I_\lambda\simeq\Fun\Hit_\cG(\bP^1)\otimes_{\bC}N_\lambda.
$$

The lemma follows immediately from the above and the following isomorphism proved in \cite[Lemma 19]{Zhu}:
$$
\Hit_\cG(\bP^1)\simeq\Hit(D_\infty)_{\fI(2)}/\Hit^\RS(D_\infty)\simeq V^*/\!\!/T.
$$ 
\end{proof}

\subsection{Proof of Proposition \ref{p:flatness}}
It suffices to show that $\gr\cD_{z_i,\lambda_i}'\simeq\otimes_i\cV_{\lambda_i}'$ is flat over $\gr A=\gr\Fun\Op_*\otimes_{\Fun V^*/\!\!/T}\Fun V^*$. Moreover, it suffices to prove this over neighborhoods $\cY$ in $T^*\Bun_\cG$ over which the vector bundles $\cV_{\lambda_i}'$ can be trivialized. From Lemma \ref{l:grFunOp*}, it suffices to show that
$\cO_{\cY}\otimes_{\bC,i}V_{\lambda_i}$ is flat over
$$
\gr\Fun\Op_*\otimes_{\Fun V^*/\!\!/T}\Fun V^*\simeq\Fun V^*\otimes_{\bC,i}N_{\lambda_i}.
$$

We know from \cite[Lemma 18]{Zhu} that the moment map $T^*\Bun_\cG\ra V^*$ is flat, so $\cO_{\cY}$ is flat over $\Fun V^*$. Also, recall that $V_{\lambda_i}$ is a rank one free $N_{\lambda_i}$-module. The desired flatness follows.

\section{Proof of Theorem \ref{t:main}}
In this section we prove Theorem \ref{t:main} following the outline described in \S\ref{ss:outline}. Let $\nabla$ be a $\LG$-connection on $\bGm$ as in the theorem.

\subsection{Generic oper structure}
In order to apply \cite[Corollary 1.1]{Arinkin}, we need to show that $\nabla$ is generically irreducible. This follows from the lemma befow:
\begin{lem}\label{l:FG irr at infty}
$\nabla_{\FG,\lambda}|_{D_\infty^\times}$ is irreducible.
\end{lem}
\begin{proof}
Denote the local monodromy group of $\nabla_{\FG,\lambda}|_{D_\infty^\times}$ by $G_{\nabla,\infty}$. We need to show $G_{\nabla,\infty}$ is not contained in any proper parabolic subgroup. From \cite[\S13]{FGr}, we know $G_{\nabla,\infty}$ is generated by the smallest torus $S$ whose Lie algebra contains $N+E\in\Lg$, together with element $n=(2\crho)(e^{\pi i/h})$ which normalizes $S$. Since $N+E$ is regular semisimple, there is a unique maximal torus $T'$ containing $S$, and $T'=C_{\LG}(S)$. Thus $n\in N_{\LG}(T')/T'$ is a Coxeter element in the Weyl group of $T'$.

Suppose $G_{\nabla,\infty}$ is contained in a parabolic subgroup $P=L_PU_P$ with unipotent radical $U_P$ and Levi subgroup $L_P$. Thus $S\subset T'\subset P$. By conjugation by $P$, we may assume $T'\subset L_P$. Thus $n\in N_{\LG}(T')\cap P$ is contained in the Weyl group of $L_P$, $n\in L_P$. Since $L_P$ is a Levi subgroup, there is a subtorus $S'\subset L_P$ such that $L_P=C_{\LG}^\circ(S')$. By conjugation, we may choose $S'\subset T'$. Thus $n\in L_P$ centralizes $S'$. However, since $n\in N_{\LG}(T')/T'$ is a Coxeter element, we have $\Lie(S')\subset(\Lie(T'))^n=0$. Thus $S'=1$, $L_P=\LG$, and the proof is complete.
\end{proof}

The above lemma implies that $\nabla|_{D_\infty^\times}\simeq\nabla_{\FG,\lambda}|_{D_\infty^\times}$ is irreducible, thus $\nabla$ is generically irreducible. By \cite[Corollary 1.1]{Arinkin}, we can equip an oper structure to $\nabla$ over some open subcurve $U\subset\bGm$. Denote $\bGm-U=\{z_1,z_2,...,z_k\}$, $S=\{0,z_1,...,z_k,\infty\}$. We obtain an oper
\begin{equation}\label{eq:generic op str}
\chi\in\Op_{\Lg}(U)
\end{equation}
whose underlying connection is $\nabla|_U$.

\subsection{Local oper structures}\label{ss:local opers}
For each $z\in S$, the restriction of $\chi$ to $D_z^\times$ defines a local oper $\chi|_{D_z^\times}\in\Op_{\Lg}(D_z^\times)$. 

\subsubsection{}
At $z=0$, the underlying connection of $\chi_0:=\chi|_{D_0^\times}\in\Op_{\Lg}(D_0^\times)$ is $\nabla|_{D_0^\times}$, which is assumed to have regular singularity with unipotent monodromy. By \cite[Lemma 10.4.2]{FrenkelBook}, there exists an integral weight $\lambda_0$ such that $\chi_0$ is contained in the subspace
$$
\chi_0\in\Op_{\Lg}^\RS(D_0)_{\varpi(-\lambda_0-\rho)}.
$$
We would also use $\chi_0$ to denote the corresponding character $\chi_0:\Fun\Op_{\Lg}^\RS(D_0)_{\varpi(-\lambda_0-\rho)}\ra\bC$.

\subsubsection{}
At each $z_i$, $i=1,...,k$, the underlying connection of $\chi_i:=\chi|_{D_{z_i}^\times}\in\Op_{\Lg}(D_{z_i}^\times)$ is $\nabla|_{D_{z_i}^\times}$, which has trivial monodromy. By \cite[Lemma 1]{FGWeyl}, there exists a unique dominant integral weight $\lambda_i$ such that $\chi_i$ is contained in the subspace
$$
\chi_i\in\Op_{\Lg}^\reg(D_{z_i})_{\varpi(-\lambda_i-\rho)}.
$$
We would also use $\chi_i$ to denote the corresponding character $\chi_i:\Fun\Op_{\Lg}^\reg(D_{z_i})_{\varpi(-\lambda_i-\rho)}\ra\bC$.

\subsubsection{}
At $z=\infty$, the underlying connection of $\chi_\infty:=\chi|_{D_\infty^\times}\in\Op_{\Lg}(D_\infty^\times)$ is $\nabla|_{D_\infty^\times}$, which is assumed to have slope $\frac{1}{h}$. By definition,
$$
\chi_\infty\in\Op_{\Lg}(D_\infty)_{\leq 1/h}.
$$
We would also use $\chi_\infty$ to denote the corresponding character $\chi_\infty:\Fun\Op_{\Lg}(D_\infty)_{\leq 1/h}\ra\bC$.

Moreover, recall $U(V)^T= U(V)\cap\fZ_{\fI(2)}\subset\End(\Vac_{\fI(2)})$ from \cite[Proposition 15]{Zhu}. Since $U(V)^T$ is just a polynomial algebra of one variable, we can lift the composition character
$$
\phi:U(V)^T\ra\fZ_{\fI(2)}\simeq\Fun\Op_{\Lg}(D_\infty)_{\leq 1/h}\xrightarrow{\chi_\infty}\bC
$$
to a character $U(V)\ra\bC$, which we still denote by $\phi$. We would also use $\phi$ to denote the corresponding additive character $\phi:V\ra\bC$. 
\begin{lem}\label{l:phi generic}
The character $\phi:V\ra\bC$ is generic, i.e. nonzero on each root subspace in $V=I(1)/I(2)$.
\end{lem}
\begin{proof}
Since $U(V)^T$ is generated by a product of dual basis of root vectors in $V$ with some multiplicities, it suffices to show $\phi:U(V)^T\ra\bC$ is nonzero.

Write the oper canonical form \eqref{eq:oper canonical form} of $\chi_\infty$ as
$$
\chi_\infty=\td+(p_{-1}+\sum_{i=1}^\ell v_i(s)p_i)\td s,\quad v_i(s)\in\bC(\!(s)\!),
$$
where $s=t^{-1}$ is a coordinate on $D_\infty^\times$, $v_i(s)=\sum_j v_{ij}s^{-j-1}$. From the slope formula in \cite[Definition 1]{CK}\footnote{Note that we use $d_i$ to denote fundamental degrees, while there $d_i$ are exponents.}, we have
$$
\sup\{0,\sup\{-\frac{\mathrm{ord}_s(v_i)}{d_i}-1\}_{i=1,...,\ell}\}=\frac{1}{h}.
$$
Thus $\mathrm{ord}_s(v_\ell)=-h-1$, $v_{\ell,h}\neq0$, i.e. $\chi_\infty(v_{\ell,h})\neq0$ for $v_{\ell,h}\in\Fun\Op_{\Lg}(D_\infty)_{\leq 1/h}$. 

On the other hand, $v_{ij}$ maps to Segal-Sugawara operator $S_{ij}\in\fZ$ under Feigin-Frenkel isomorphism. Denote $\pi:\fZ\twoheadrightarrow\fZ_{\fI(2)}\simeq\Fun\Op_{\Lg}(D_\infty)_{\leq 1/h}$, then $\chi_\infty(\pi(S_{\ell,h}))\neq 0$. Let $d=h$ in \cite[Proposition 28.(i)]{KXY}, we get $\pi(S_{\ell,h})\in U(V)^T$\footnote{This result is originally due to an unpublished work of Tsao-Hsien Chen.}\footnote{Although in \cite{KXY} they use a potentially different set of generators $S_{ij}$ of the center $\fZ$ and assumes $G$ is a classical group, their argument in the case of $d=h$ works for any set of generators $S_{ij}$ and any simple reductive group.}.Thus $\phi(\pi(S_{\ell,h}))=\chi_\infty(\pi(S_{\ell,h}))\neq0$, which proves the lemma.
\end{proof}

\subsection{The automorphic sheaf}
Let $\Loc$ be the multiple points localization functor for $S=\{0,z_1,...,z_k,\infty\}$ as defined in \S\ref{s:flatness}. Consider the following D-module on $\Bun_\cG$:
\begin{equation}\label{eq:cA Loc}
\cA:=\omega_{\Bun_\cG}^{-1/2}\otimes\Loc(\bM_{w_0\lambda_0}^\opp/\ker\chi_0\otimes_{\bC,i}\bV_{\lambda_i}/\ker\chi_i\otimes_{\bC}\Vac_{\fI(2)}/\ker\phi),
\end{equation}

where $\ker\chi_i\subset\fZ$ and $\ker\phi\subset U(V)$ act on the corresponding $\hg$-modules. 

On the other hand, we can see from the discussion in \S\ref{ss:local opers} that
\begin{align*}
\chi\in\Op_*:=&\Op_{\Lg}(U)\times_{\Op_{\Lg}(D_0^\times)}\Op_{\Lg}^\RS(D_0)_{\varpi(-\lambda_0-\rho)}\times_{i,\Op_{\Lg}(D_{z_i}^\times)}\Op_{\Lg}^\reg(D_{z_i})_{\varpi(-\lambda_i-\rho)}\\
&\times_{\Op_{\Lg}(D_\infty^\times)}\Op_{\Lg}(D_\infty)_{\leq 1/h}.
\end{align*}
where the only difference of the above from \eqref{eq:Op_*} is that here $w_0\lambda_0$ is replaced by $w_0(w_0\lambda_0)=\lambda_0$. Since $\chi_0,\chi_i,\chi_\infty$ in \eqref{eq:cA Loc} are all restrictions of $\chi$, we obtain from \eqref{eq:global loc} that
\begin{equation}\label{eq:cA global}
\cA=\omega_{\Bun_\cG}^{-1/2}\otimes((\cD_{\Bun_\cG}'\otimes\cL_{w_0\lambda_0}\otimes_i\cV_{\lambda_i})\otimes_{\Fun\Op_*\otimes_{U(V)^T}U(V)}\bC_{\chi,\phi})
\end{equation}

where $\bC_{\chi,\phi}$ is the character of $\Fun\Op_*\otimes_{U(V)^T}U(V)$ defined from $\chi,\phi$.

\begin{prop}\label{p:cA properties}
$\cA$ is a D-module on $\Bun_{\cG}$ that is
\begin{enumerate}
	\item [(i)] nonzero,
	\item [(ii)] holonomic,
	\item [(iii)] $(V,\cL_\phi)$-equivariant, where $\cL_\phi$ is the pullback of the exponential D-module via $\phi:V\ra\bGa$,
	\item [(iv)] a Hecke eigensheaf with eigenvalue $\nabla|_U$,
	\item [(v)] irreducible.
\end{enumerate} 
\end{prop}
\begin{proof}
(i): This is immediate from Proposition \ref{p:flatness}.

(ii): It suffices to prove the holonomicity over local neighborhoods $\cY$ in $\Bun_{\cG}$ where the bundles $\cL_{w_0\lambda_0}$, $\cV_{\lambda_i}$ can be trivialized. From \eqref{eq:cA global}, we can see $\cA|_{\cY}$ is a quotient of a direct sum of 
$$
\omega_{\cY}^{-1/2}\otimes(\cD_{\cY}'\otimes_{U(V)}\bC_\phi)=(\omega_{\Bun_{\cG}}^{-1/2}\otimes(\cD_{\Bun_{\cG}}'\otimes_{U(V)}\bC_\phi))|_{\cY}.
$$
From \cite[Lemma 18]{Zhu}, we know $\omega_{\Bun_{\cG}}^{-1/2}\otimes(\cD_{\Bun_{\cG}}'\otimes_{U(V)}\bC_\phi)$ is holonomic. Thus $\cA|_{\cY}$ is holonomic.

(iii): The action of $V$ on $\Bun_{\cG}$ induces $U(V)\ra\End(\cD_{\Bun_{\cG}}')\ra\End(\cD_{z_i,\lambda_i}')$, thus being $(V,\cL_\phi)$-equivariant is clear from construction.

(iv): Provided the flatness in Proposition \ref{p:flatness}, the proof is the same as the proof of \cite[Corollary 9]{Zhu}. Note that we only obtain an eigenvalue over $U=\bGm-\{z_1,...,z_k\}$, because at $z_i$'s we apply the localization functor to quotients of Weyl modules rather than quotients of vacuum module.

(v): From (i)-(iii), we know $\cA$ is a nonzero holonomic $(V,\cL_\phi)$-equivariant D-module on $\Bun_{\cG}$, i.e. $\cA\in D_h\!-\!\mathrm{mod}(\Bun_\cG)^{(V,\cL_\phi)}$. From the discussion in \S\ref{sss:review on FG}and \eqref{eq:HNY main}, this implies that $\cA$ is a direct sum of the unique irreducible object in $D_h\!-\!\mathrm{mod}(\Bun_\cG)^{(V,\cL_\phi)}$. From the discussion on the last page of \cite{Zhu} under Lemma 21, we know this irreducible object is $\omega_{\Bun_{\cG}}^{-1/2}\otimes(\cD_{\Bun_{\cG}}'\otimes_{U(V)}\bC_\phi)$. We need to show $\cA$ is one copy of this D-module. It suffices to show this for the associated graded of the restriction of $\cA$ to the open substack $V\hookrightarrow\Bun_{\cG}$. Denote $p:T^*V\ra V$. For the irreducible object, note that $\omega_{\Bun_{\cG}}^{-1/2}$ is trivializable over affine space $V$. We have
$$
\gr(\omega_{\Bun_{\cG}}^{-1/2}\otimes(\cD_{\Bun_{\cG}}'\otimes_{U(V)}\bC_\phi))|_V
\simeq \gr(\cD_V\otimes_{U(V)}\bC_\phi)
\simeq p_*(\cO_{T^*V}\otimes_{S(V)}\bC_{\bar{\phi}})
\simeq p_*(\cO_V),
$$
where $\bar{\phi}$ is the associated graded of $\phi$, $V$ is the zero section $V\hookrightarrow V\times\Lie(V)^*\simeq T^*V$.

For $\cA$, note that $\cL_{w_0\lambda_0},\bV_{\lambda_i}$ can be trivailized over $V$. By Lemma \ref{l:grFunOp*}, we have
\begin{align*}
\gr(\cA|_V)
&\simeq \gr(\cD_V\otimes_{\bC,i}V_{\lambda_i})\otimes_{\Fun\Op_*\otimes_{U(V)^T}U(V)}\bC_{\chi,\phi}\\
&\simeq
p_*((\cO_{T^*V}\otimes_{\bC,i}V_{\lambda_i})\otimes_{U(V)\otimes_{\bC,i}N_{\lambda_i}}\bC_{\bar{\chi},\bar{\phi}})\\
&\simeq
p_*(\cO_V\otimes_{\bC,i}(V_{\lambda_i}\otimes_{N_\lambda}\bC_{\bar{\chi_i}})).
\end{align*}
Since $V_{\lambda_i}$ is a rank one free module over $N_{\lambda_i}$, we get $\gr(\cA|_V)\simeq p_*(\cO_V)$. This implies
\begin{equation}\label{eq:cA isom}
\cA\simeq\omega_{\Bun_{\cG}}^{-1/2}\otimes(\cD_{\Bun_{\cG}}'\otimes_{U(V)}\bC_\phi),
\end{equation}
which completes the proof.
\end{proof}

\subsection{Identification of eigenvalues}
From Proposition \ref{p:cA properties}, we know $\nabla|_U$ is eigenvalue of irreducible Hecke eigensheaf $\cA$. From \eqref{eq:cA isom} and the last page of \cite{Zhu}, we know the eigenvalue of $\cA$ is also the following Frenkel-Gross connection:
$$
\nabla_\phi:=\td+(N+\phi(\pi(S_{\ell,h}))tE)\frac{\td t}{t},
$$
where $\pi(S_{\ell,h})$ is a generator of $U(V)^T$ defined in the proof of Lemma \ref{l:phi generic}, and $\phi(\pi(S_{\ell,h}))\neq 0$ from that lemma. Thus we obtain an isomorphism $\nabla|_U\simeq\nabla_\phi|_U$. Since $\nabla$ and $\nabla_\phi$ are both regular on $\bGm$, this isomorphism extends to an isomorphism $\nabla\simeq\nabla_\phi$ over $\bGm$, which completes the proof of Theorem \ref{t:main}.

\end{document}